\documentclass[preprint]{imsart}

\usepackage{amsthm,amsmath}
\usepackage{amsfonts,amstext,amssymb, mathtools, comment}

\startlocaldefs
\usepackage[utf8x]{inputenc}
\usepackage{float}
\usepackage{amsmath}
\usepackage{amsfonts,amstext,amssymb, mathtools, comment}
\usepackage{graphicx,epsfig}
\usepackage{caption}
\usepackage{subcaption}

\newcommand{\p}{{\mathbb P}}
\newcommand{\e}{{\mathbb E}}
\newcommand{\D}{{\mathrm d}}

\renewcommand{\a}{\alpha}

\newcommand{\1}[1]{\mbox{\rm 1}_{\{#1\}}}

\newcommand{\bs}{\boldsymbol}
\newcommand{\ba}{{\bs \nu}}
\newcommand{\be}{{\bs e}}

\newcommand{\bt}{{\bs t}}
\newcommand{\one}{{\bs 1}}
\newcommand{\matI}{\mathbb{I}}

\renewcommand{\l}{{\bf WRONG}}
\newcommand{\z}{{\zeta}}
\newtheorem{theorem}{Theorem}

\newtheorem{prop}{Proposition}
\newtheorem{remark}{Remark}
\newtheorem{alg}{Algorithm}
\newtheorem{example}{Example}

\endlocaldefs

\begin{document}

\begin{frontmatter}
\title{Exact boundaries in sequential testing for phase-type distributions}
\runtitle{Sequential testing for phase-type distributions}


\begin{aug}
\author{\fnms{Hansj{\"o}rg} \snm{Albrecher}},
\author{\fnms{Peiman} \snm{Asadi}}
\and
\author{\fnms{Jevgenijs} \snm{Ivanovs}}
\address{Department of Actuarial Science, University of Lausanne, CH-1015 Lausanne, Switzerland}

\runauthor{Albrecher, Asadi and Ivanovs}
\end{aug}

\begin{abstract}
 We consider Wald's sequential probability ratio test for deciding whether a sequence of independent and identically distributed observations
comes from a specified phase-type distribution or from an exponentially tilted alternative distribution. In this setting,
we derive exact decision boundaries for given Type I and Type II errors by establishing a link with ruin theory.
Information on the mean sample size of the test can be retrieved as well. The approach relies on the use of matrix-valued 
scale functions associated to a certain one-sided Markov additive process. By suitable transformations the results also apply to 
other types of distributions including some distributions with regularly varying tail.
\end{abstract}

\end{frontmatter}

\section{Introduction}
Consider Wald's sequential probability ratio test \cite{wald} of a simple hypothesis against a simple alternative.
Let $\z_1,\z_2,\ldots$ be a sequence of independent and identically distributed random variables (observations) with density $f$,
where either $f=f_0$ (hypothesis $H_0$) or $f=f_1$ (hypothesis $H_1$).
The log-likelihood ratio $\Lambda_k$ for the first $k$ observations is then given by
\begin{align*}
 &\Lambda_k=\sum_{i=1}^k\log\frac{f_0(\z_i)}{f_1(\z_i)},\hspace{0.2in}\Lambda_0=0,
\end{align*}
and its first exit time $N$ from the interval $(a,b)$ by
\begin{equation}\label{eq:N}
N=\inf\{k\geq 0:\Lambda_k\notin (a,b)\},
\end{equation}
where $a<0<b$. At time $N$ the sampling is stopped and a decision is made:
accept $H_0$ if $\Lambda_N\geq b$, and accept $H_1$ if $\Lambda_N\leq a$. The corresponding errors are given by
$\a_0=\p_0(\text{reject }H_0)$ and $\a_1=\p_1(\text{reject }H_1),$ where $\p_i$ indicates that hypothesis $H_i$ is valid.

One now wants to choose decision boundaries $a$ and $b$ so that the errors are below prespecified thresholds.
If it is possible to find $a$ and $b$, such that the errors coincide with their respective thresholds, then Wald's test with such boundaries is known to be optimal (i.e.\
the expected number of observations (under both hypotheses) is minimal) among all tests respecting these thresholds, see~\cite{wald_wolfowitz_optimal} and~\cite[Thm.~IV.4]{shirjaev}.
Such a pair $(a,b)$ of boundaries is unique under very weak assumptions~\cite{weiss_unique}, which do hold in our setting below.
Usually, a pair $(a,b)$ resulting in the prespecified errors exists, unless the problem is `too easy',
in which case an optimal test will use zero observations with positive probability, cf.~\cite{wijsman} for an analysis of a more general test.

The following simple bounds on the decision boundaries are known, see~\cite{wald}:
\begin{align}\label{eq:wald}
& a\geq \log \frac{\alpha_1}{1-\alpha_0}, \quad b\leq \log
\frac{1-\alpha_1}{\alpha_0}.
\end{align}
In practice these bounds are often used as actual decision boundaries. As a result, $N$ increases and one of the errors may surpass its threshold, however usually not by a large amount for small errors, see~\cite{wald}.

If the errors $\alpha_0$ and $\alpha_1$ can be determined for any fixed pair $(a,b)$, then the optimal decision boundaries can be found by a numerical search for any given pair of errors $\alpha_0,\alpha_1$ of interest. This inverse problem is however hard even for simple cases. Some tractable examples can be found in Wald~\cite{wald} and~Teugels \& Van Assche 
\cite{teugels_SPRT}, where the latter assume $f_0$ and $f_1$
to be densities of exponential distributions. 
Some strong asymptotic results were obtained in~\cite{lotov},
but they still require identification of the Wiener-Hopf factors corresponding to the random walk $\Lambda_k$,
which can be done explicitly only in some cases.

In the present work we assume that $f_0$ and $f_1$ are densities of phase-type distributions where one
can be obtained by exponential tilting of the other. This includes the case of two exponential densities, as well as two Erlang densities with identical shape parameter. After translating the inverse problem of Wald's test into a two-sided exit problem embedded in classical ruin theory (Section \ref{sec2}), we use techniques for Markov additive processes (Section \ref{sec3}) to establish a surprisingly simple identity, which leads to explicit formulas in Section \ref{sec4}.  The approach simplifies the proof for the exponential case developed in \cite{teugels_SPRT} and extends the results to phase-type densities (taking monotone transformations of the original observations, the results are also applicable for other distributions, such as distributions with regularly varying tails obtained from exponentiating phase-type random variables). In Section \ref{sec:Erlang} we discuss the Erlang case in more detail, for which a very explicit treatment is possible. Section \ref{sec6} provides a general formula for the expected number of observations in Wald's test. Section \ref{sec:var_bayes} studies the uniqueness issue further and considers an extension to a Bayesian version, where an a priori probability for the correctness of $H_0$ is available. Finally Section \ref{sec:numerics} provides some numerical illustrations.

\section{Wald's test and ruin theory}\label{sec2}
Let $f_0$ be a probability density function of some positive random variable $\z$, and let $\p_0$ be the corresponding probability measure.
Consider the Laplace-Stieltjes transform $G_0(\theta)=\e_0 e^{-\theta \z},\theta\geq 0$ of $\z$ and
define a new tilted measure $\p_1$ according to $\frac{\D \p_1}{\D \p_0}=\frac{1}{G_0(\theta)}e^{-\theta \z}$. Then, under $\p_1$, $\z$ has a probability density function $f_1$
given by
\begin{equation}\label{eq:f1}
 f_1(x)=\frac{1}{G_0(\theta)}e^{-\theta x}f_0(x).
\end{equation}

Consider Wald's test for densities $f_0$ and $f_1$, where $\theta>0$, and observe that
\[\log\frac{f_0(x)}{f_1(x)}=\theta x+\log G_0(\theta).\]
Hence the log-likelihood ratio $\Lambda_k$ is a random walk with increments distributed as $\theta \z -d$, where $d=-\log G_0(\theta)>0$ and $\z$ has density $f_0$ (under $H_0$) or $f_1$ (under $H_1$). Define the closely related continuous-time stochastic process
\begin{equation}
\label{eq:risk_proc}X_t=\theta t-\sum_{i=1}^{N_t} d,\quad t\geq 0,
\end{equation}
 where $N_t$ is a renewal process with inter-arrival times distributed as $\z$, see Figure~\ref{fig:risk_process}.
\begin{figure}[h!]
\begin{center}
\includegraphics{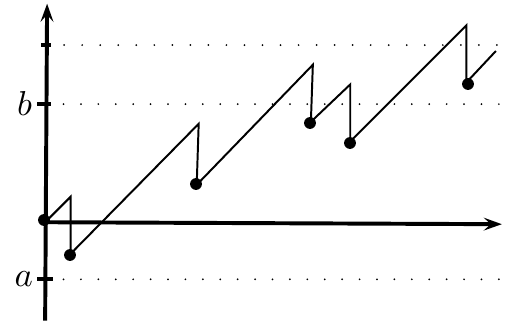}
\caption{Sparre Andersen risk process}
\label{fig:risk_process}
\end{center}
\end{figure}
One can interpret $X_t$ as a surplus process of an insurance portfolio under a Sparre Andersen risk model with initial capital 0, where premiums are collected 
at constant rate $\theta$, and claims of (deterministic) 
size $d$ arrive according to the renewal process $N_t$ (see e.g. \cite{asm_ruin}).
Importantly, one can recover the random walk $\Lambda_k$ from the continuous-time process $X_t$ by considering it at the epochs of jumps. Letting
\begin{align*}
&\tau_a^-=\inf\{t\geq 0:X_t\leq a\},\quad \tau_b^+=\inf\{t\geq 0:X_t\geq b\}
\end{align*}
for $a<0<b$, we observe that
\begin{align}\label{eq:errors}
 &\a_0=\p_0(\tau_a^-<\tau_{b+d}^+),\quad \a_1=\p_1(\tau_{b+d}^+<\tau_a^-),
\end{align}
which is an artifact of the deterministic jumps. Thus we have arrived at a pair of two-sided exit problems for the risk process $X_t$ -- one under $H_0$ and the other one under $H_1$.

\section{Phase-type distributions and Markov additive processes}\label{sec3}
In this section we present a solution of the two-sided exit problem for the process $X_t$ under the assumption that the generic interarrival time $\z$ has a phase-type (PH) distribution,
i.e.\ the distribution of the life time of a transient continuous-time Markov Chain (MC) on finitely many states $1,\ldots,n$, see e.g.~\cite{asm_ruin}.
A PH distribution is parametrized by the transition rate matrix $T$ of the corresponding MC and the row vector $\ba$ representing the initial distribution.
Denoting by $\bt=-T\one$ the column vector of killing (absorption) rates, one can express the density of $\z$ as
\begin{equation}\label{eq:PH}
f(x)=\ba e^{Tx}\bt.
\end{equation}
The Erlang distribution of rate $\lambda$ is retrieved for $\ba=(1,0,\ldots,0)$ and choosing $T$ as a square matrix with $-\lambda$ on the main diagonal, $\lambda$ on the upper diagonal, and 0 elsewhere. Note that the class of PH distributions is dense in the class of all distributions on $(0,\infty)$.

Consider now a bivariate process $(X_t,J_t)$, where $X_t$ is the risk process defined in~\eqref{eq:risk_proc}, and $J_t$ tracks the phase of the current interarrival time, which has PH distribution.
It is not hard to see that $J_t$ is a MC with transition rate matrix $T+\bt\ba$, i.e.\ the transitions can happen due to phase change
or due to arrival of a claim (kill and restart). Furthermore, $(X_t,J_t)$ is a simple example of a Markov Additive Process (MAP) without positive jumps, see~\cite{asm_ruin} for a definition.
Such a process is characterized by a matrix-valued function $F(s),s\geq 0$, which satisfies $\e(e^{s X_t}1_{\{J_t=j\}}|J_0=i)=(e^{F(s)t})_{ij}$ for all $t\geq 0$ and $i,j\in\{1,\ldots,n\}$.
In our case we have the identity
\begin{equation}\label{eq:F}
F(s)=T+\theta s\matI+\bt\ba e^{-ds},
\end{equation}
where $\matI$ is the identity matrix. The diagonal elenments $\theta s$ represent the linear evolution of $X_t$ with slope $\theta$ (the same value in every phase)
and $\bt\ba$ is a matrix of transition rates of $J_t$ causing the jump in $X_t$ with transform $e^{-ds}$,
see~\cite[Prop.~4.2]{asm_ruin}.

The two-sided exit problem for MAPs without positive jumps was solved in~\cite{ivanovs_scale}, 
and the solution resembles the one for a L\'evy process without positive jumps~\cite[Thm.\ 8.1]{kyprianou}.
According to ~\cite{ivanovs_scale}, the matrix of probabilities with $ij$th element $\p(\tau_b^+<\tau_a^-,J_{\tau_b^+}=j|J_0=i)$ is given by $W(-a)W(-a+b)^{-1}$,
where $W(x),x\geq 0$ is a continuous matrix-valued function (called \textit{scale function}) characterized by the transform
\begin{equation}\label{eq:transform}
 \int_0^\infty e^{-sx}W(x)\D x=F(s)^{-1}
\end{equation}
for $s$ large enough. It is known that $W(x)$ is non-singular for $x>0$ and so is $F(s)$ in the domain of interest. Since $J_0$ has distribution $\ba$, we write
\begin{equation}\label{eq:exit}
\p(\tau_b^+<\tau_a^-)=\ba W(-a)W(-a+b)^{-1}\one,
\end{equation}
with $\one=(1,\ldots,1)$. Note that the scale function is given in terms of its transform, and the only known explicit examples assume that all jumps of $X_t$ have PH distributions. In the present setting the jumps are not PH but deterministic,
which nevertheless gives some hope for the inversion problem. Indeed, in the case of an Erlang distribution for $\zeta$ we obtain an explicit representation of $W(x)$, see Section~\ref{sec:Erlang}.

\section{Identification of the errors}\label{sec4}
In the following we assume that $f_0$ is a density of a PH distribution with parameters $T_0,\ba_0$ and $\bt_0=-T_0\one$, see~\eqref{eq:PH}.
Its transform is known to be $G_0(\theta)=\ba_0(\theta\matI-T_0)^{-1}\bt_0$. Consider the density $f_1$, defined in~\eqref{eq:f1},
of the corresponding exponentially tilted distribution with the tilt parameter $\theta>0$.
In~\cite{asm_tilting} it is shown that this tilted distribution is again PH, and the parameters are given by
\begin{equation}\label{eq:PH1}
 T_1=\Delta^{-1} T_0\Delta-\theta \matI,\quad \ba_1=\ba_0\Delta/G_0(\theta),\quad\bt_1=\Delta^{-1}\bt_0,
\end{equation}
where $\Delta$ is a diagonal matrix with $(\theta\matI-T_0)^{-1}\bt_0$ on the diagonal. These  diagonal elements are all in $(0,1)$, which can be seen from the representation of $G(\theta)\in(0,1)$ for different initial distributions $\ba_0$.

Since both $f_0$ and $f_1$ correspond to PH distributions, we can combine~\eqref{eq:errors} and~\eqref{eq:exit} to obtain
\begin{align}\label{eq:errorsW01}
&\a_0=\p_0(\tau_a^-<\tau_{b+d}^+)=1-\ba_0 W_0(-a)W_0(-a+b+d)^{-1}\one,\nonumber \\
&\a_1=\p_1(\tau_{b+d}^+<\tau_a^-)=\ba_1 W_1(-a)W_1(-a+b+d)^{-1}\one,
\end{align}
where $W_0(x)$ and $W_1(x)$ are the (matrix-valued) scale functions corresponding to the MAP $(X_t,J_t)$ for $f=f_0$ and $f=f_1$, respectively.
Interestingly, $W_0$ and $W_1$ are intimately related:

\begin{prop}\label{prop:Wrel}
The scale functions $W_0(x)$ and $W_1(x)$ satisfy
\[W_1(x)=e^{x}\Delta^{-1} W_0(x)\Delta\]
for all $x\geq 0$.
\end{prop}
\begin{proof}
We establish that
\[\Delta^{-1} F_0(s-1)\Delta=\Delta^{-1} (T_0-\theta\matI+\theta s\matI+\bt_0\ba_0/G_0(\theta)e^{-ds})\Delta=
T_1+\theta s\matI+\bt_1\ba_1 e^{-ds}=F_1(s),\]
by using~\eqref{eq:F},~\eqref{eq:PH1} and recalling that $d=-\log G_0(\theta)$.
Now we can check that the proposed matrix-valued function indeed gives the desired transform, see~\eqref{eq:transform}:
\[\int_0^\infty e^{-sx}\left(e^{x}\Delta^{-1} W_0(x)\Delta\right)\D x=\Delta^{-1}\int_0^\infty e^{-(s-1)x}W_0(x)\D x\Delta=\Delta^{-1}F_0(s-1)^{-1}\Delta=F_1(s)^{-1}\]
for $s$ sufficiently large. The result follows, because the transform identifies the continuous $W_1(x),x\geq 0$ uniquely.
\end{proof}
\begin{remark}
This curious relation -- revealed by an application in sequential testing -- would be hard to 
obtain by simple tailoring of parameters - the corresponding quantities simplify in an intriguing  way. It also paves the way for further interesting relations between the two processes, which, however, are outside the scope of the present paper.
\end{remark}

Combining~\eqref{eq:errorsW01} and Proposition~\ref{prop:Wrel}, we obtain the following result.
\begin{theorem}\label{thm:main}
Let $f_0$ be a density of a PH distribution with parameters $T_0,\ba_0,\bt_0$, and $f_1$ be the corresponding exponentially tilted density with the tilt parameter $\theta>0$.
 The errors $\a_0$ and $\a_1$ corresponding to the decision boundaries $a<0<b$ in the Wald test of $f_0$ against $f_1$ are given by
\begin{align*}
&\a_0=1-\ba_0 W_0(-a)W_0(-a+b+d)^{-1}\one,\nonumber \\
&\a_1=e^{-b}\ba_0 W_0(-a)W_0(-a+b+d)^{-1}(\theta\matI-T_0)^{-1}\bt_0,
\end{align*}
where $d=-\log G_0(\theta)>0$, $G_0(\theta)$ is the Laplace transform of $f_0$, and the continuous matrix-valued function $W_0(x),x\geq 0,$ is identified by
\[\int_0^\infty e^{-sx}W_0(x)\D x=\left(T_0+\theta s\matI+\bt_0\ba_0 e^{-ds}\right)^{-1}\] for large $s$.
\end{theorem}

The transform of $W_0(x)$ can be inverted in certain cases. In Section~\ref{sec:Erlang} we provide
an explicit expression of $W_0(x)$ when $f_0$ (and then also $f_1$) is the density of an Erlang distribution. In other cases one can use numerical methods.

In addition, Theorem~\ref{thm:main} provides simple bounds for the level $b$.
First, observe that $\ba_0 W_0(-a)W_0(-a+b+d)^{-1}$ is a vector of probabilities, and recall that all the entries of $(\theta\matI-T_0)^{-1}\bt_0$ are in~$(0,1)$.
Then we can write
\[m(1-\a_0) \leq \a_1e^b\leq (1-\a_0)M,\]
where $m$ and $M$ are the minimal and the maximal entries of $(\theta\matI-T_0)^{-1}\bt_0$. Hence also
\begin{equation}\label{eq:bound}
 \log\frac{1-\a_0}{\a_1}+\log m\leq b\leq \log\frac{1-\a_0}{\a_1}+\log M,
\end{equation}
where both $\log m$ and $\log M$ are negative. This provides an improvement (for the PH case) of the widely used general Wald bound $b\leq \log\frac{1-\a_0}{\a_1}$.

\begin{example}
If $f_0$ is the density of an exponential distribution with rate $\lambda_0$, then $f_1$ is a density of an exponential distribution with rate $\lambda_1=\lambda_0+\theta$. Here the matrix $T_0$ reduces to a scalar $-\lambda_0$,
and hence $m=M=\lambda_0/(\lambda_0+\theta)=\lambda_0/\lambda_1$ leading to $b=\log\frac{1-\a_0}{\a_1}-\log\frac{\lambda_1}{\lambda_0}$. This simple identity
for exponential densities was already established in~\cite{teugels_SPRT}. Computation of the boundary $a<0$ is more involved, and relies on the identity
\[W_0(-a)/W_0(-a+\log\frac{1-\a_0}{\a_1})=1-\a_0,\]
where $W_0(x)$ will be identified in Section \ref{sec:Erlang}.
\end{example}
In general, we do not have a closed form solution for $b$, and hence the two equations in Theorem~\ref{thm:main} need to be solved simultaneously.

\section{Erlang against Erlang}\label{sec:Erlang}
Throughout this section we consider the case when $f_0$ is the density of an Erlang distribution with $n$ phases and rate $\lambda_0$,
i.e. $f_0(x)=\lambda_0^nx^{n-1}e^{-\lambda_0 x}/(n-1)!$, which has Laplace transform $G_0(\theta)=\left(\frac{\lambda_0}{\lambda_0+\theta}\right)^n$.
Exponential tilting of $f_0$ with the tilt parameter $\theta>0$ results in $f_1$, which is another Erlang density on $n$ phases, but with rate $\lambda_1=\lambda_0+\theta$.
Hence our setup allows to consider two arbitrary Erlang distributions with the same number of phases.

Under the Erlang assumption, the jump size $d=-n\log (\lambda_0/\lambda_1)$ only depends on the ratio $\rho=\lambda_0/\lambda_1$ of the two rates, not on their absolute values. Also, since 
$\theta\,\cdot$ Erlang($n,\lambda_i$)$\sim$Erlang($n,\lambda_i /\theta$), a scaling of $\theta$ down to 1 simply stretches the process $X_t$ of \eqref{eq:risk_proc} in the horizontal direction by the factor $\theta$ (under both hypotheses) and the law of the random walk $\Lambda_k$ is unchanged. Hence Wald's test only depends on the ratio $\rho$ and w.l.o.g. we can choose $\theta=1$, i.e. $\lambda_1=\lambda_0+1$, leading to 
$\lambda_0=\rho/(1-\rho)$ and $\lambda_1=1/(1-\rho)$ for the ratio $\rho=\lambda_0/\lambda_1\in (0,1)$.

%
%
Consider the PH parameters $T_0,\ba_0$ and $\bt_0$ of the density $f_0$, where $T_0$ is an $n\times n$ matrix with $-\lambda_0$ on the diagonal, $\lambda_0$ on the upper diagonal and 0 elsewhere;
$\ba_0=(1,0,\ldots,0)=\be_1$ and $\bt_0=(0,\ldots,0,\lambda_0)'$. Some algebraic manipulations show that the vector
$(\theta\matI-T_0)^{-1}\bt_0$ simplifies to $(\rho^{n},\rho^{n-1},\ldots,\rho^{1})'$, and so by Theorem~\ref{thm:main} we have
\begin{align}\label{eq:errors_Erlang}
&\a_0=1-\be_1 W_0(-a)W_0(-a+b+d)^{-1}\one, \\
&\a_1=e^{-b}\be_1 W_0(-a)W_0(-a+b+d)^{-1}(\rho^{n},\rho^{n-1},\ldots,\rho^{1})',\nonumber
\end{align}
where $d=-n\log\rho$, the transform of $W_0(x)$ is given by $F_0(s)^{-1}$, and according to~\eqref{eq:F}
\begin{equation}\label{eq:FErlang}
F_0(s)=\begin{pmatrix}
                   s-\lambda_0&\lambda_0&0&\ldots&0\\
            0&s-\lambda_0&\lambda_0&\ldots&0\\
&&\ldots&&\\
            \lambda_0 e^{-sd}&0&0&\ldots&s-\lambda_0\\
                  \end{pmatrix}
\end{equation}
for $n\geq 2$, whereas $F_0(s)=s-\lambda_0+\lambda_0 e^{-sd}$ for $n=1$.
The bounds~\eqref{eq:bound} for $b$ now simplify to
\begin{equation}\label{imprbou}\log\frac{1-\a_0}{\a_1}-n\log\rho^{-1}\leq b\leq \log\frac{1-\a_0}{\a_1}-\log \rho^{-1},\end{equation}
where $\log\rho^{-1}>0$.
It turns out that $W_0(x)$ has a relatively simple expression as a sum of $\lfloor x/d\rfloor$ terms.
\begin{theorem}\label{thm:W}
 Consider a MAP with $n$ phases characterized by $F_0(s)$ given in~\eqref{eq:FErlang} for an arbitrary $d>0$.
Then the \text{ij}th element of the scale function $W_0(x)$ for $x\geq 0$
is given by
\begin{equation}\label{qua}
 W_0(x)_{ij}=\sum_{k=\1{i>j}}^{\lfloor x/d\rfloor} g(\lambda_0(x-dk),kn+j-i),\qquad i,j=1,\ldots,n 
\end{equation}
where $g(y,m)=\frac{(-y)^m}{m!}e^{y}$. 
\end{theorem}
\begin{proof}
In the proof we drop the subscript $0$. We need to invert the transform $\int_0^\infty e^{-s x}W(x)\D x=F(s)^{-1}$.
Application of Cramer's rule and careful computation of co-factors yields
\[\left(F(s)^{-1}\right)_{ij}=\frac{1}{(s-\lambda)^n-(-\lambda)^ne^{-sd}}\times\begin{cases}
                                                                (-\lambda)^{l}(s-\lambda)^{n-l-1}, &i\leq j\\
(-\lambda)^{l}(s-\lambda)^{n-l-1} e^{-sd}, &i> j\\
                                                               \end{cases},
\]
where $l=(j-i)\mod n$.
Note that the fraction in front can be written as $(s-\lambda)^{-n}\sum_{k=0}^\infty (-\lambda)^{kn}(s-\lambda)^{-kn}e^{-sdk}$ for sufficiently large $s$. Hence 
\begin{align*}
\left(F(s)^{-1}\right)_{ij}=\sum_{k=0}^\infty \frac{(-\lambda)^{kn+l}}{(s-\lambda)^{kn+l+1}}e^{-sd(k+\1{i>j})}.
\end{align*}
Using $\int_0^\infty e^{-sx}\frac{x^n}{n!}e^{\lambda x}\D x=\frac{1}{(s-\lambda)^{n+1}}$ we invert $\frac{(-\lambda)^{kn+l}}{(s-\lambda)^{kn+l+1}}$ to obtain $\frac{(-\lambda x)^{kn+l}}{(kn+l)!}e^{\lambda x}=g(\lambda x,kn+l)$.
The factor $e^{-sd(k+\1{i>j})}$ amounts to shifting $x$ to $x-d(k+\1{i>j})$.
Hence for $j\geq i$
\[W(x)_{ij}=\sum_{k=0}^\infty g(\lambda(x-dk),kn+j-i)\1{x\geq dk}=\sum_{k=0}^{\lfloor x/d\rfloor} g(\lambda(x-dk),kn+j-i).\]
Similarly, for $j<i$ we have
\[W(x)_{ij}=\sum_{k=0}^\infty g(\lambda(x-d(k+1)),kn+n+j-i)\1{x\geq d(k+1)}=\sum_{k=1}^{\lfloor x/d\rfloor} g(\lambda(x-dk),kn+j-i),\]
which concludes the proof.
\end{proof}
The quantity \eqref{qua} is intimately connected with the waiting time distribution in an $E(n)/D/1$ queue, see for instance \cite{franx}. In a risk theory context, for the case $n=1$ (which refers to the compound Poisson model with deterministic jumps), formula \eqref{qua} can already be found in \cite[Sec.3.3.2.1]{segerdahl}, see also \cite{gerber}.

\section{On the number of observations}\label{sec6}
In this section we determine $\e z^N$ under both hypotheses, where $N$ is the number of observations leading to a decision, see~\eqref{eq:N}. To that end, some further exit theory of MAPs~\cite{ivanovs_scale} can be used (and the present context provides an interesting 
illustration of the applicability of the latter). We will also utilize the concept of killing, see e.g.~\cite{ivanovs_killing}.

Suppose we kill our MAP $(X_t,J_t)$ right before every jump $-d$ with probability $1-z$, where $z\in(0,1]$ (i.e., the process is sent to an additional absorbing state).
Write $\p^z$ for the corresponding probability measure.
Then \[\p^z(\tau_a^-<\tau_{b+d}^+)=\e(z^N \1{\Lambda_N\leq a}),\] because the process has to survive $N$ independent killing instants.
Similarly,
\[z\p^z(\tau_{b+d}^+<\tau_a^-)=\e(z^N\1{\Lambda_N\geq b}),\]
where prefactor $z$ comes from the fact that the MAP should not be killed at the jump following its first passage time over $b+d$.
Adding these two equations we obtain $\e z^N$.

Importantly, all exit identities still hold for the killed MAP, which is characterized by $F^z(s)=T+\theta s\matI+\bt\ba ze^{-ds}$.
In particular, $\p^z(\tau_{b+d}^+<\tau_a^-)=\ba W^z(-a)W^z(-a+b+d)^{-1}\one$, where the transform of the scale function $W^z(x)$ evaluates to $F^z(s)^{-1}$.
Furthermore, from Corollary 3 in~\cite{ivanovs_scale} we have
\[\p^z(\tau_a^-<\tau_{b+d}^+)=\ba\left(Z^z(-a)-W^z(-a)W^z(-a+b+d)^{-1}Z^z(-a+b+d)\right)\one,\]
where $Z^z(x)=\matI-\int_0^xW^z(y)\D yF^z(0)$. Therefore,
\[\e z^N=\ba\left(Z^z(-a)-W^z(-a)W^z(-a+b+d)^{-1}Z^z(-a+b+d)+zW^z(-a)W^z(-a+b+d)^{-1}\right)\one.\]
Noting that $F^z(0)\one=(z-1)\bt$, differentiating with respect to $z$ and letting $z\uparrow 1$ we get
\begin{equation}\label{eq:EN}
 \e N=-\int_0^{-a}\ba W(y)\bt\D y+\ba W(-a)W(-a+b+d)^{-1}\left(\int_0^{-a+b+d} W(y)\bt\D y+\one\right),
\end{equation}
where $W(x)$ corresponds to the case of no killing ($z=1$).
Here we also used differentiability of $W^z(x)$ and $Z^z(x)$ in $z$, which can be shown using further fluctuation identities.
Formula \eqref{eq:EN} provides both $\e_0 N$ and $\e_1 N$, where the latter can be expressed through the quantities associated to hypothesis $H_0$ using Proposition~\ref{prop:Wrel} and~\eqref{eq:PH1}.

\section{Variational and Bayesian formulation}\label{sec:var_bayes}
\subsection{Variational formulation: the optimality region}
So far we have focused on the variational formulation of Wald's test. According to this formulation, for given errors $\a_0$ and $\a_1$ one needs to
determine the decision boundaries $a<0<b$ resulting in these errors. For that purpose one can solve the two equations of Theorem~\ref{thm:main} using numerical methods.
When such boundaries exist, they are unique and they define the optimal test minimizing both $\e_0 N$ and $\e_1 N$.
The following algorithm can be used to determine the region $R$ of $(\a_0,\a_1)$ in $[0,1]\times [0,1]$, for which the decision boundaries (resulting in the errors) exist,
and hence Wald's test is optimal. This algorithm can be analyzed using monotonicity results from~\cite{wijsman_monotonicity}. We omit its thorough discussion.
\begin{alg}
Determination of the optimality region $R$:
\begin{enumerate}
 \item Find the errors $\a_0^*$ and $\a_1^*$ corresponding to $a=b=0$.
\item Fix $b=0$; for all $\a_0$ in $[0,\a_0^*)$ determine $a$ which results in $\a_0$ and then find the corresponding $\a_1>\a_1^*$.
\item Fix $a=0$; for all $\a_1$ in $[0,\a_1^*)$ determine $b$ which results in $\a_1$ and then find the corresponding $\a_0>\a_0^*$.
\end{enumerate}
These two continuous curves $(\a_0,\a_1)$, the point
$(\a_0^*,\a_1^*)$, and the axis provide the boundary of the
optimality region $R$. \label{alg1}
\end{alg}
We provide an example for the optimality region $R$ in Section~\ref{sec:numerics}.
It indicates that $R$ is large enough to include most cases of practical interest. If the pair of errors lies outside of $R$, then
the problem of testing is `too easy', i.e.\ a certain test, which uses zero observations with positive probability, will perform better than any Wald's test with $a<0<b$.

\subsection{Bayesian formulation}
In the Bayesian formulation, it is assumed that $H_0$ has some prior probability $\pi\in[0,1]$, see e.g.~\cite{shirjaev}.
For fixed constants $c,c_0,c_1>0$ one defines a penalty (or average loss)
\begin{equation}\label{eq:penalty}
 \gamma=\pi(c \e_0 N+c_0\a_0)+(1-\pi)(c \e_1 N+c_1\a_1),
\end{equation}
which is to be minimized.
It turns out that there always exists a test which is optimal for all $\pi$, i.e. it minimizes the penalty among all tests.
The rule is to stop when the posterior probability of $H_0$ exits some interval $(a^*,b^*)$, where $0\leq a^*\leq b^*\leq 1$, with the obvious decision.
Expressing the posterior probability through $\Lambda_k$, one observes that an equivalent rule is to stop when $\Lambda_k$ exits
\begin{equation}\label{eq:bayes}
 (a,b)=(\log\frac{a^*}{1-a^*} + \log\frac{1-\pi}{\pi},\log\frac{b^*}{1-b^*} + \log\frac{1-\pi}{\pi}),
\end{equation}
see~\cite{shirjaev}. 
Recall that for a given pair $(a,b)$ we can find $\a_0,\a_1$ and $\e_0 N,\e_1 N$ using Theorem~\ref{thm:main} and \eqref{eq:EN} respectively,
and so we can calculate the penalty~$\gamma$ for a fixed prior $\pi$. Hence to find an optimal $(a,b)$,
corresponding to the minimal penalty, we only need to run a numerical optimization routine. If this $(a,b)$ is the unique pair minimizing the penalty, then $(a^*,b^*)$ can be recovered from the above relation. 
\section{Numerical illustrations}\label{sec:numerics}
In this section we provide an illustration of the applicability of our results for both the variational and Bayesian formulation.
For simplicity we choose an Erlang distribution with 2 phases and rate $\lambda$, 
and consider Wald's test of the simple hypothesis $\lambda=\lambda_0$ against the simple alternative $\lambda=\lambda_1$, where $\lambda_0<\lambda_1$. In Section~\ref{sec:Erlang} it was shown that in such a situation Wald's test depends only on the single parameter $\rho=\lambda_0/\lambda_1\in(0,1)$ and  the scale function $W_0(x)$ has an explicit representation.\\
\indent Let us first consider the variational formulation. We choose errors $\a_0=0.05$ and $\a_1=0.025$, and find 
the decision boundaries $a<0<b$ by solving~\eqref{eq:errors_Erlang} numerically. Figure~\ref{fig:Bounds} depicts $a$ and $b$ as functions of $\rho$ (solid lines), 
as well as their Wald bounds~\eqref{eq:wald} (dashed lines) and the improved upper and lower bounds for $b$ from \eqref{imprbou} (dotted lines). 
Figure~\ref{fig:EN} depicts $\max(\e_0 N,\e_1 N)$ for the exact boundaries (solid line)
and their Wald bounds (dashed line), respectively.
\begin{figure}[H]
        \centering
        \begin{subfigure}[b]{0.4\textwidth}
                \centering
                \includegraphics[trim = 40mm 85mm 40mm 80mm, clip, width=7cm]{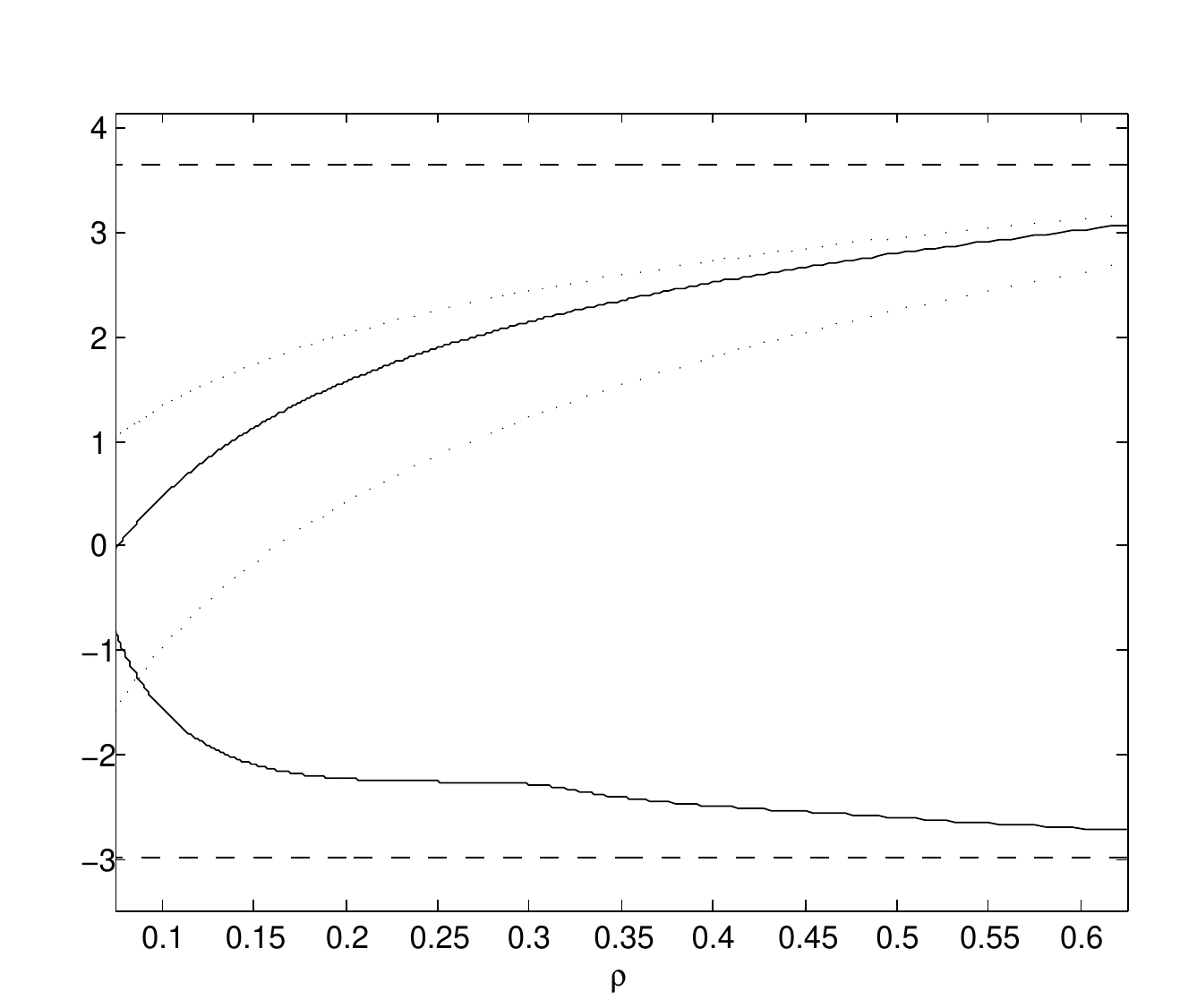}
                \caption{Decision boundaries and their bounds}
                \label{fig:Bounds}
        \end{subfigure}%
        ~ 
        \begin{subfigure}[b]{0.4\textwidth}
                \centering
                \includegraphics[trim = 40mm 85mm 40mm 80mm, clip, width=7cm]{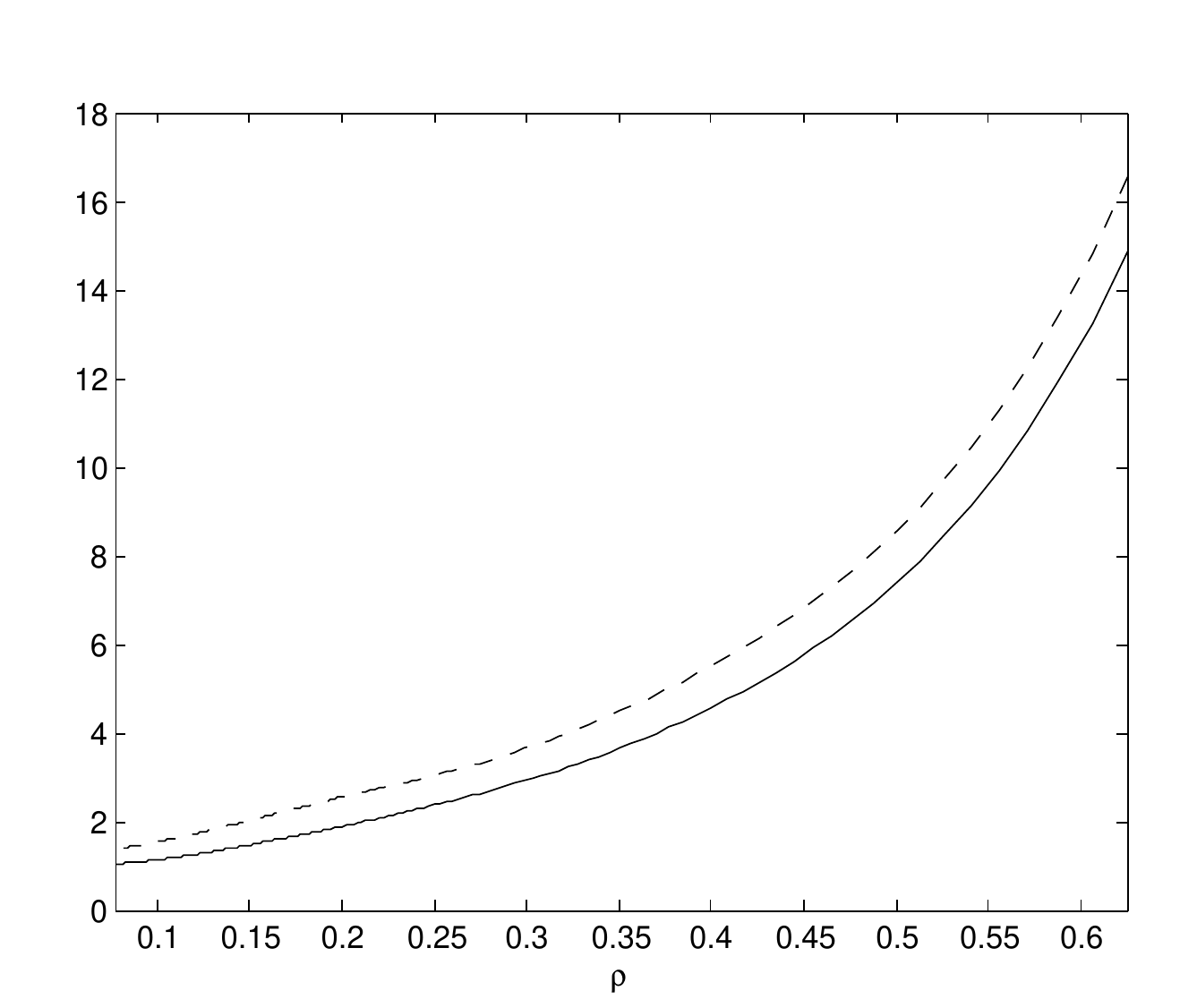}
                \caption{$\max(\e_0 N,\e_1 N)$}
                \label{fig:EN}
        \end{subfigure}
        ~ 

        \caption{Decision boundaries and the maximal expected number of observations}\label{fig:BoundsEN}
\end{figure}
Let us briefly comment on the case when $\rho$ is close to $1$, i.e. the test problem is very hard.
In this case the increments of the random walk $\Lambda_k$ decrease in absolute value. 
This implies that $\Lambda_N$ is very close to $a$ or $b$ (depending on the side of exit), which makes the Wald bounds very tight (see also a discussion in~\cite{wald}). In Figures~\ref{fig:Bounds} and ~\ref{fig:EN} one can see that the boundaries get indeed closer to their Wald bounds and the expected number of observations increases as $\rho\rightarrow 1$. When $\rho$ gets close to $1$, also numerical problems arise, as due to small $d$ the number of terms in the representation of $W_0(x)$ becomes large (cf. Theorem~\ref{thm:W}).\\
\indent On the other hand, when $\rho$ decreases to 0, the test problem becomes simpler. When one of the boundaries hits 0, the Wald test stops being optimal (cf. Algorithm~\ref{alg1}).
Figure~\ref{fig:OptimalRegion} depicts optimality regions of the Wald test for different values of  $\rho$ for the above Erlang(2) example.\\
\indent Let us turn our attention now to the Bayesian formulation, see Section~\ref{sec:var_bayes}.
We choose $c=0.1,c_0=1$ and $c_1=2$ for the penalty $\gamma$ in ~\eqref{eq:penalty} and two different values $\pi=0.3$ and $\pi=0.7$ for the prior.
Figure~\ref{fig:bayes_ab} depicts the optimal boundaries $a$ and $b$ (minimizing the penalty). 
These boundaries are used to compute the optimal boundaries $a^*$ and $b^*$ for the posterior probability by virtue of ~\eqref{eq:bayes}, 
which can only be done if $(a,b)$ is a unique pair achieving the minimal penalty. The result is depicted in Figure~\ref{fig:bayes_abstar}.
Recall that the latter boundaries do not depend on the prior $\pi$, and hence the lines for both $\pi$ should coincide.
This is indeed the case up to $\rho\approx 0.38$, at which point $a$ (corresponding to $\pi=0.3$) hits level 0 and uniqueness is lost (in this case, $b$ can be any positive number).
The correct values of $a^*$ and $b^*$ follow the solid lines corresponding to $\pi=0.7$.\\
\indent Note that the behavior of the boundaries $a$ and $b$ as functions of $\rho$ is substantially  different for the variational and the Bayesian formulation.
For increasing $\rho$, the distance between the decision boundaries increases in the former case and decreases in the latter, where controlling the number of observations becomes the dominant issue. 

\begin{figure}
        \centering
        \includegraphics[trim = 40mm 85mm 40mm 80mm, clip, width=7cm]{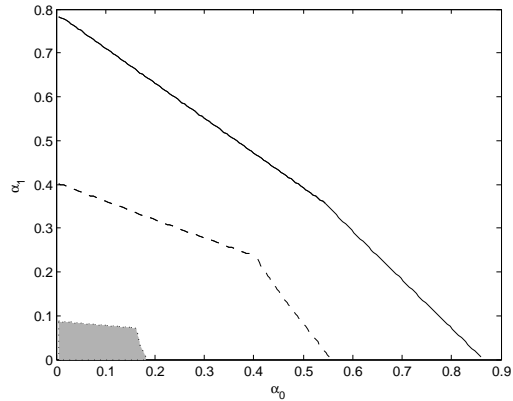}
        \caption{Optimality region $R$ for $\rho=1/6$ (shaded), $\rho=1/2$ (below dashed line) and $\rho=5/6$ (below solid line)}
        \label{fig:OptimalRegion}
\end{figure}
\begin{figure}[H]
        \centering
        \begin{subfigure}[b]{0.4\textwidth}
                \centering
                \includegraphics[trim = 40mm 85mm 40mm 80mm, clip, width=7cm]{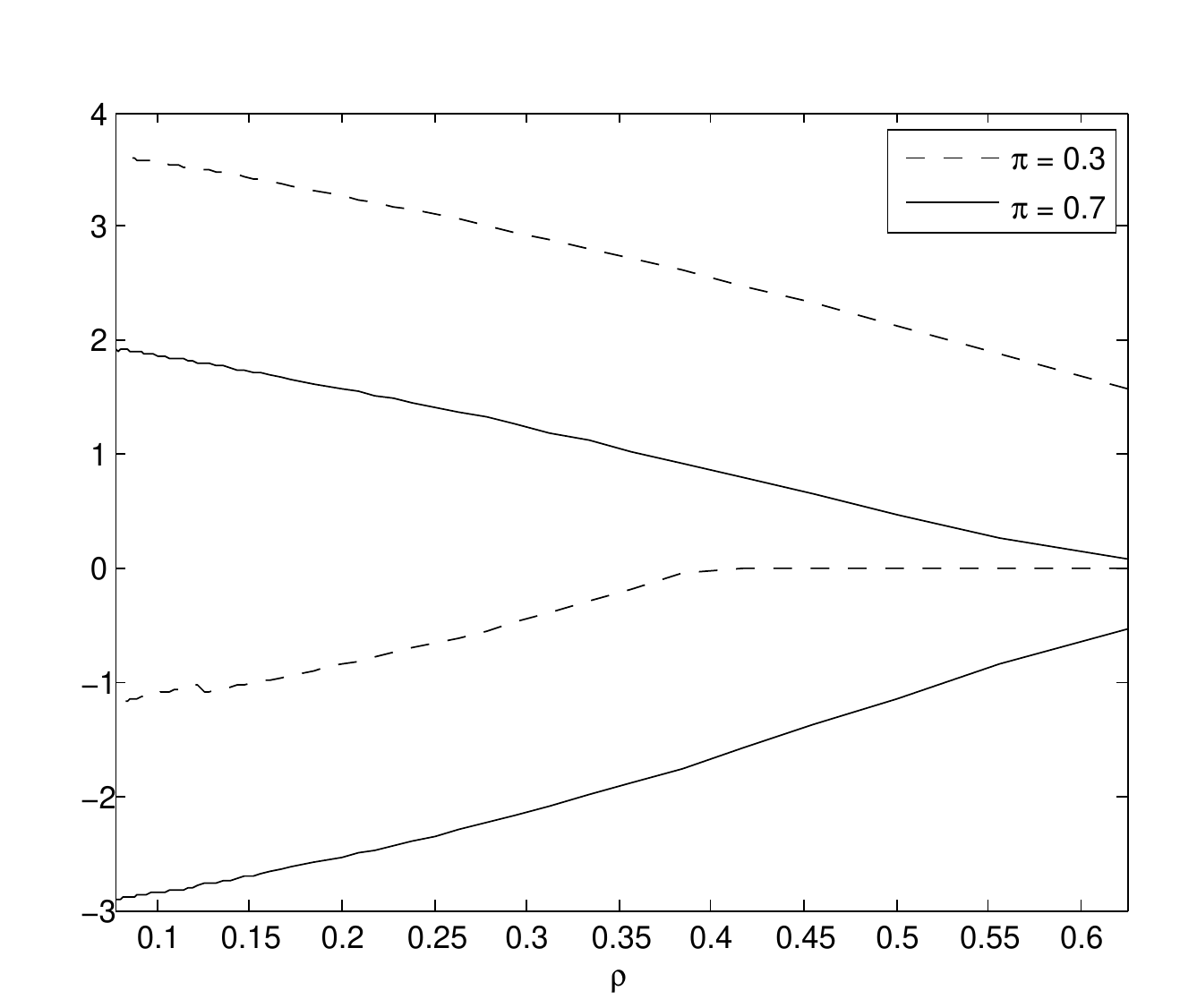}
                \caption{Decision boundaries $a$ and $b$ for $\Lambda_k$}
                \label{fig:bayes_ab}
        \end{subfigure}%
        ~ 
        \begin{subfigure}[b]{0.4\textwidth}
                \centering
                \includegraphics[trim = 40mm 85mm 40mm 80mm, clip, width=7cm]{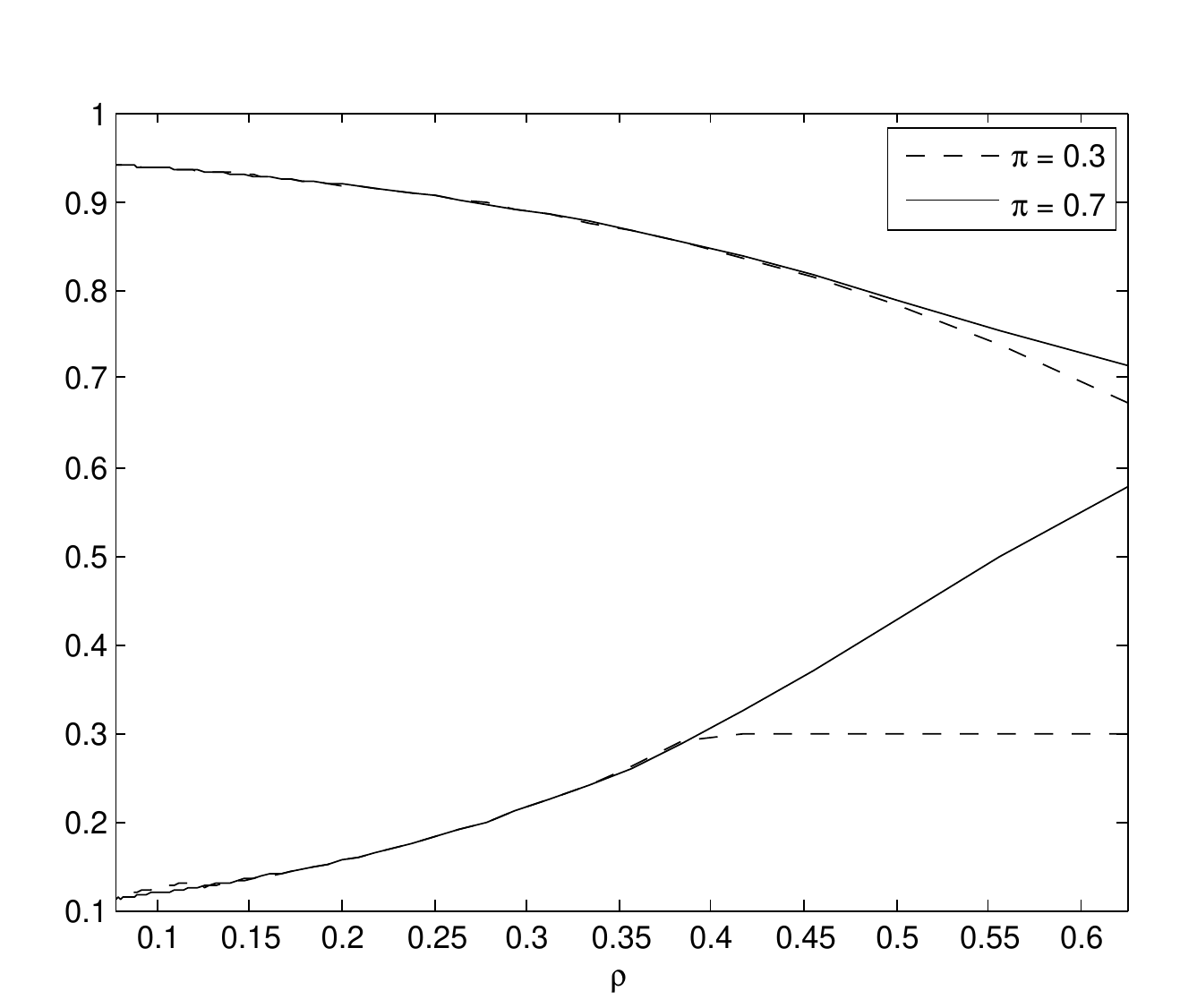}
                \caption{Boundaries $a^*$ and $b^*$ for the posterior}
                \label{fig:bayes_abstar}
        \end{subfigure}
        ~ 

        \caption{Decision boundaries for the Bayesian formulation}
\end{figure}

\section{Acknowledgments}
We would like to thank Onno Boxma, Dominik Kortschak, Andreas L\"{o}pker and Jef Teugels for their valuable comments. This work was supported by the Swiss National Science Foundation Project 200020-143889 and the EU-FP7 project Impact2C.


\bibliographystyle{abbrv}
\bibliography{./sprt}
\end{document}